\documentclass{amsart}%
\usepackage{amsmath}
\usepackage{amsfonts}
\usepackage{amssymb}
\usepackage{graphicx}%
\usepackage{color}
\usepackage{amsmath}
\usepackage{amsfonts}
\usepackage{amssymb}
\usepackage{graphicx}%
\usepackage{xypic}
\usepackage{color}
\providecommand{\U}[1]{\protect\rule{.1in}{.1in}}
\newtheorem{theorem}{Theorem}

\newtheorem{corollary}[theorem]{Corollary}

\newtheorem{example}[theorem]{Example}

\newtheorem{lemma}[theorem]{Lemma}

\newtheorem{proposition}[theorem]{Proposition}
\newtheorem{remark}[theorem]{Remark}

\begin{document}

\title{ Commutativity of central sequence algebras }

\date{\today}

\author{Dominic Enders}
\author{Tatiana Shulman}





\begin{abstract}
The question of which separable C*-algebras have abelian central sequence algebras was raised and studied by Phillips (\cite{Phillips}) and Ando-Kirchberg (\cite{AndoKirchberg}). In this paper we give a complete answer to their question:

\medskip

A separable C*-algebra $A$ has abelian central sequence algebra if and only

if $A$ satisfies Fell's condition.

\medskip

\noindent Moreover, we introduce a higher-dimensional analogue of Fell's condition and show that it completely characterizes subhomogeneity of central sequence algebras.

\noindent In contrast, we show that any non-trivial extension by compact operators has not only non-abelian but not even residually type I central sequence algebra. In particular its central sequence algebra is not type I and not residually finite-dimensional (RFD).

\noindent Our techniques extensively use properties of nilpotent elements in C*-algebras.
\end{abstract}

\maketitle

\section*{Introduction}
Studying the center is a key step in (C*-)algebra structure theory.
In the context of operator algebras one further looks at central sequences, i.e. bounded sequences $(x_n)$ of elements of a C*-algebra $A$ satisfying $$[x_n, a]\to 0, \;\text{for any}\; a\in A.$$ (Here for von Neumann algebras one naturally means the convergence with respect to tracial norms, while for C*-algebras it is with respect to the C*-norm).

This notion of central sequences has its origin in von Neumann algebra theory and turned out to be extremely fruitful there.
In fact, the central sequence algebra was a crucial tool in several breakthrough results in von Neumann algebra theory and we refer the reader to \cite{Ioana} and the references therein for a more detailed discussion on this topic.
In the context of this paper, a special focus was put on identifying which $II_1$-factors have trivial central sequence algebra (these are  $II_1$-factors without the so-called property Gamma) or which have abelian central sequence algebra (those for which it is not abelian are known to be the McDuff factors).

On the C*-algebra side, the study of central sequences was initiated shortly after the success of their von Neumann algebra counterparts.
Akemann and Pedersen proved that a C*-algebra $A$ has trivial central sequence algebra if and only if $A$ is a continuous trace algebra (\cite{AkemannPedersen}). The question of when the central sequence algebra of a separable C*-algebra is abelian also arose, and many significant partial results were obtained in \cite{Phillips}.

Two decades later, the concept of central sequence algebras gained momentum again, starting from the seminal work of Kirchberg \cite{Kirchberg}  -- this time in connection with the classification program for simple nuclear C*-algebras.
In particular, the concept of strongly self-absorbing C*-algebras (\cite{TomsWinter}) and the tensorial absorption of such in terms of embeddings into central sequence algebras have become an indispensable tool in the classification program.

Progress on the question of when a central sequence algebra is abelian appeared in the work of Ando and Kirchberg (\cite{AndoKirchberg}) where they proved that for any non-type I C*-algebra its central sequence algebra is non-abelian.
Ando and Kirchberg further provided examples (similar to the ones constructed  in \cite{PhDthesis}) showing that central sequence algebras of  type I C*-algebras can be both abelian and non-abelian.
However, the general question of which C*-algebras have abelian central sequence algebra remained open.

In this paper we give a complete answer to this question. More generally, we characterize those C*-algebras whose central sequence algebra are subhomogeneous.\\

\noindent {\bf Main Theorem} Let $A$ be a separable C*-algebra and $k\in \mathbb N$. Then the central sequence algebra  $F(A)$ of $A$  is $k$-
subhomogeneous if and only if $A$ satisfies Fell's condition of order $k$.\\

A C*-algebra satisfies Fell's condition if each of its irreducible representations $\pi_0$ satisfies Fell's condition, meaning that there exist $b\in  A^+$ and an open neighbourhood $U$ of $\pi_0$ in $\hat A$ such that $\pi(b)$ is a rank-one projection whenever $\pi \in U$ (e.g. see \cite{Blackadar}, \cite{FellProperty}).
The class of C*-algebras satisfying Fell's condition is known to coincide with the class of type $I_0$-algebras (\cite{FellProperty}) which were introduced by Pedersen  as C*-algebras generated by their abelian elements (see \cite{Blackadar}).
In this paper we introduce Fell's condition of order $n$ and the type $I_n$-algebras as a natural generalization of Fell's condition and type $I_0$-algebras respectively. To obtain our characterization of C*-algebras with $n$-subhomogeneous central sequence algebra, we prove that these conditions are all equivalent (Proposition 16).

\medskip

The proof of the Main Theorem above is obtained from 3 auxiliary steps listed below, each requiring its own set of techniques.
These techniques, while being quite different, have one feature in common -- they all extensively use properties of nilpotent elements in C*-algebras.
The reason being that commutativity (and similarly subhomogenity) can efficiently be captured in terms of the existence of nilpotent elements with prescribed order.

\medskip

\begin{itemize}
\item Step 1: Show that if $A$ is a C*-algebra of operators strictly containing all compact operators, then $F(A)$ is not $n$-subhomogeneous for any $n\in \mathbb N$.
\end{itemize}

\medskip

\noindent For $n=1$ and unital $A$, this step was done by Phillips (\cite{Phillips}).
 Here, in fact, we obtain a far reaching generalization of this result:\\

 \noindent {\bf Theorem 9.} Let $A \subset B(H)$ be a separable C*-algebra such that $A \supsetneq K(H)$. Then $F(A)$ is not residually type I.\\

\noindent Our proof is also more direct in the sense that it doesn't require Connes technique for automorphism groups or the connection between automorphism groups and central sequences in general, which was the main tool in \cite{Phillips}.

\medskip

\begin{itemize}
\item Step 2: to prove that if $A$ is CCR and $F(A)$ is $n$-subhomogeneous, then $A$ satisfies Fell's condition of order $n$.
\end{itemize}

\medskip

\noindent In this step we show that without a suitable bound on the local rank of elements in $A$ (local in the sense of the spectrum $\hat{A}$), it is possible to exhibit approximately central, nilpotent elements of high order.
 This is achieved by pulling back nilpotent elements through high-dimensional representations while preserving approximate centrality.

\medskip

\begin{itemize}
\item Step 3: Show that if $A$ is type $I_{n-1}$, then $F(A)$ is $n$-subhomogeneous.
\end{itemize}

\medskip

\noindent In this step we make heavy use of liftability of nilpotent elements.
This way we can transfer information about commutativity / subhomogenity from the central sequence algebra $F(A)$ back to the algebra $A$ itself.
Furthermore, we exploit stability of nilpotent elements to squeeze asymptotically central nilpotent elements into subhomogeneous subalgebras which eventually leads to the desired dimensional constraints.

\medskip

The paper is organized as follows.
In Section 1 we give preliminaries on central sequences and necessary information on properties of nilpotent elements in C*-algebras.
Section 2 is dedicated to Step 1, i.e. studying central sequences in extensions of compact operators. In particular, Theorem 9 mentioned above is proved there.
In section 3 we introduce Fell's condition of higher rank and type $I_n$-algebras and prove the equivalence of Fell's condition of order $n$ and of being of type $I_{n-1}$.
In Section 4 we do Step 2, that is we prove that if $A$ is CCR and $F(A)$ is $n$-subhomogeneous, then $A$ satisfies Fell's condition of order $n$.
In Section 5 we do Step 3, that is we prove that if $A$ is type $I_{n-1}$, then  $F(A)$ is $n$-subhomogeneous.
In Section 6 we prove the Main Theorem and deduce its consequences for automorphisms and derivations of C*-algebras satisfying Fell's condition.

\medskip

\textbf{Acknowledgements.}
The work of the second-named author was supported by the Swedish Research Council and by the Polish National Science Centre grant under the contract number 2019/34/E/ST1/0017.

\section{Preliminaries}

\subsection{Central sequence algebras}
 A central sequence in a C*-algebra is a bounded
 sequence which commutes asymptotically with each element of the algebra. "Asymptotically" is meant with respect to a given filter; in particular  in \cite{Kirchberg} it was meant with respect to a fixed free ultrafilter, while  in earlier works (e.g. \cite{AkemannPedersen}, \cite{Phillips})
 one used the usual convergence (that is the filter of all cofinite sets).

 Fix a filter $\omega$ on $\mathbb N$.
 Let $A$ be a C*-algebra. We denote by $A_{\infty}$ the C*-algebra of all bounded sequences with entries in $A$. Elements of $A_{\infty}$ will be written as $(x_n)$.
 The ideal $c_{\omega}$ of $A_{\infty}$ is defined as
 $$c_{\omega} = \{(x_n)\in A_{\infty}\;|\; \|x_n\|\to_{\omega} 0\}.$$ The {\it C*-ultraproduct algebra} $A_{\omega}$ is defined by $A_{\omega} = A_{\infty}/c_{\omega}.$ The canonical surjection $A_{\infty} \to A_{\omega}$ will be denoted by $\pi_{\omega}$. The elements $a\in A$ will be naturally identified with the images of the constant sequences $(a, a, \ldots)$ in $A_{\omega}$.
 Let $A'\bigcap A_{\omega}$ be the relative commutant of $A$ in $A_{\omega}$, i.e. elements of $A'\bigcap A_{\omega}$ are represented by {\it central sequences}, that is sequences $(x_n)$ satisfying $$[x_n, a]\to_{\omega} 0, \;\text{for any}\; a\in A.$$
 Furthermore, consider the annihilator of $A$ in $A_\omega$, i.e.
 $$Ann(A, A_{\omega}) = \{\pi_{\omega}((x_n))\;|\; \|x_na\|+\|ax_n\|\to_{\omega} 0, \;\text{for any }\;a\in A\},$$
 which is an ideal in $A_\omega$.
 Dividing out this ideal we obtain the {\it central sequence algebra} $F(A, \omega)$:  $$F(A, \omega) = (A'\bigcap A_{\omega})/Ann(A, A_{\omega})$$
 Note that for unital $A$ we have $Ann(A,A_\omega)=0$, so that the definition of $F(A, \omega)$ reduces to $F(A, \omega)=A'\bigcap A_\omega$ in this case.

 \medskip

 In  \cite{Phillips}, for a unital C*-algebra $A$, a central sequence $(x_n)$ in $A$ was called {\it hypercentral}  if
$$[x_n, y_n]\to 0$$ for any central sequence $(y_n)$ in $A$.
 This definition can be extended to the non-unital case as follows: a central sequence $(x_n)$ is hypercentral if $$\|[x_n, y_n]a\| + \|a[x_n, y_n]\|\to 0,$$ for any central sequence $(y_n)$ and any $a\in A$ (see also \cite{Lupini} where such sequences were called strict-hypercentral).

 \medskip

 The following question was first raised in \cite{Phillips}: $$\text{When are all central sequences hypercentral?}$$
 Clearly this is a question of when $F(A, \omega)$ is abelian where $\omega$ is the filter of all cofinite sets.
 Not surprisingly, this question turns out to be equivalent to the main question of \cite{AndoKirchberg},  asking when $F(A, \omega)$ is abelian, where $\omega$ is a free ultrafilter.
 The following proposition shows the equivalence of these two questions, as well as a quantitative version thereof.

 \begin{proposition}\label{epsilon-delta} Let $A$ be a  separable C*-algebra and let $A_1$ denote its unit ball. Let $\omega$ be a filter on $\mathbb N$ containing the filter of all cofinite sets.  The following are equivalent:
 \medskip

	 1) $F(A, \omega)$ is abelian;
\medskip


	 2) For any finite set $F\subset A_1$ and $\epsilon>0$ there exist a finite set $G\subset A_1$ and $\delta>0$ such that
	 		\[
		\text{
			$\begin{array}{l}
			\max_{a\in G}\|[a,x]\|<\delta \\
			\max_{a\in G}\|[a,y]\|<\delta
			\end{array}
		\Rightarrow
		 \|[x,y]f\|<\epsilon$ for all $f\in F$}.
		\]
	\medskip
Consequently, if $F(A, \omega)$ is abelian for one such  filter, then it is abelian for any such filter.
 \end{proposition}
 \begin{proof}
 (1)$\Rightarrow$(2): By contradiction. Assume otherwise, i.e. assume there exists a finite set $F\subset A_1$ and $\epsilon >0$  such that for all $n\in \mathbb N$ there are $x_n, y_n \in A_1$ with
 \begin{equation}\label{ImpliesConvergence} \max_{a\in G_n}\|[a, x_n]\|< 1/n \;\text{ and}\; \max_{a\in G_n}\|[a, y_n]\|< 1/n\end{equation} (where
 $G_1\subset G_2\subset \ldots$ is an increasing sequence of finite subsets of $A_1$ with dense union) but such that
 $\|[x_n, y_n]f\|\ge \epsilon$ for some $f\in F$ and all $n$. From (\ref{ImpliesConvergence}) we find that $[x_n, a]\to 0$ and $[y_n, a]\to 0$ for any $a\in A$. Since $\omega$ contains the filter of all cofinite sets, it implies that $[x_n, a]\to_{\omega} 0$ and $[y_n, a]\to_{\omega} 0$.
 Then writing $x= (x_1, x_2, \ldots)$ and $y= (y_1, y_2, \ldots)$, we find $x, y\in A_{\infty}$ with $\pi_{\omega}(x), \pi_{\omega}(y)\in A_{\omega}\bigcap A'$, but
 $$\|[\pi_{\omega}(x), \pi_{\omega}(y)]f\|=\lim_{n\to \omega}\|[x_n, y_n]f\|\ge \epsilon,$$
 i.e. $[\pi_{\omega}(x), \pi_{\omega}(y)]\notin Ann(A, A_{\omega})$ and hence the images of $x$ and $y$ in $F(A, \omega)$ do not commute.

 (2)$\Rightarrow$(1): Let $x= (x_1, x_2, \ldots)$ and $y= (y_1, y_2, \ldots)$ be two contractions in $A_{\infty}$ such that $\pi_{\omega}(x), \pi_{\omega}(y)\in A_{\omega}\bigcap A'.$ Fix $f\in A$. Let $\epsilon >0$ and choose $G$ and $\delta$ with respect to $F= \{f\}$ and $\epsilon$.  For $a\in A$, let $E_a = \{n\in \mathbb N\;|\; \|[a, x_n]\|< \delta\}$ and $E'_a = \{n\in \mathbb N\;|\; \|[a, y_n]\|< \delta\}$. Since  $[a, x_n]\to_{\omega} 0$ and $[a, y_n]\to_{\omega} 0$, we have  $E_a, E'_a\in {\omega}$. Let $E = \bigcap_{a\in G} (E_a\bigcap E'_a)$. Then
 $\{n\in \mathbb N\;|\; \|[x_n, y_n]f\|< \epsilon\} \supseteq E\in \omega.$ Therefore $\|[x_n, y_n]f\|\to_{\omega}0.$ Since it holds for any $f\in A$, we conclude that $[\pi_{\omega}(x), \pi_{\omega}(y)]\in Ann(A, A_{\omega}).$ Thus the images of $x$ and $y$ in $F(A, \omega)$ commute.

 \end{proof}

 From now we will write $F(A)$ instead of $F(A, \omega)$.

\subsection{Nilpotents in C*-algebras}

An element $x\in A$ is {\it nilpotent of order $n$}  if $x^n=0$. It was discovered in \cite{OlsenPedersen} that a nilpotent element in any quotient C*-algebra can be lifted to a nilpotent of the same order. In \cite{NilpotentContraction} it was proved that this can even be done in a norm preserving way.

\begin{theorem}\label{LiftabilityOfNilpotents} (Liftability of nilpotent contractions) Let $I$ be an ideal in a C*-algebra $A$ and let $\pi:A\to A/I$ be the canonical surjection. Let $x\in A/I$ be such that $x^n=0$ and $\|x\|\le 1$. Then there exists $a\in A$ such that $\pi(a) =x, a^n=0,\;\text{and}\; \|a\|\le 1.$
\end{theorem}

An important consequence is the stability of nilpotents under small perturbations.

\begin{corollary}\label{StabilityOfNilpotents} (Stability of nilpotent contractions) Given $n\in \mathbb N$ and $\epsilon >0$, there exists $\delta$ such that the following holds: for any C*-algebra $A$ and any $x\in A$ satisfying $\|x^n\|\le \delta$ and $\|x\|\le 1$ there is $y\in A$ such that
$y^n=0$, $\|y\|\le 1$ and $\|y-x\|\le \epsilon.$
\end{corollary}

Recall that a C*-algebra is {\it $n$-subhomogeneous} if the dimension of each its irreducible representations is not larger than $n$.
The following characterization of $n$-subhomogenity is obtained in \cite{HadwinCharacterizationSubhomogeneous}.

\medskip

\begin{theorem}\label{Fact} A C*-algebra $A$ is $n$-subhomogeneous if and only if each nilpotent element in $A$ has order not larger than $n$.
\end{theorem}

For reader's convenience we will give here an alternative proof of it.

\medskip

{\it Proof}: Suppose $A$ is not $n$-subhomogeneous.
If $A$ is type I, there exists an irreducible representation $\rho$ on a Hilbert space  $H$ of dimension larger than $n$ (possibly infinite) such that $\rho(A)\supseteq K(H)$. Since there exists an operator $T\in K(H)$ such that $T^{n+1}=0$ but $T^n\neq 0$ we can apply Theorem \ref{LiftabilityOfNilpotents} and lift $T$ to a nilpotent element of order $n+1$ in $A$.
If $A$ is not type I, we find the CAR-algebra as a subquotient of $A$ by Glimm's theorem (\cite{Glimm}) and an element $T$ therein with $T^{n+1}=0$ but $T^n\neq 0$.
Again, we can lift $T$ to a nilpotent element of order $n+1$ in $A$ just as before.

The opposite implication is clear: if $A$ is n-subhomogeneous and $a\in A$ is a nilpotent, then under each irreducible representation $\rho(a)^n=0$, which implies $a^n=0$. \qed

\begin{corollary}
A C*-algebra is commutative if and only if it does not contain any non-trivial nilpotent elements.
\end{corollary}

\bigskip

For a Hilbert space $H$, let $B(H)$ denote the C*-algebra of all bounded operators on $H$ and $K(H)$ the ideal of all compact operators.
Recall that a C*-algebra $A$ is {\it CCR} if for any  its irreducible representation $\pi$ on a Hilbert space $H$,
$$\pi(A) = K(H).$$ If for any   irreducible representation $\pi$ of $A$ on a Hilbert space $H$,
$$\pi(A) \supseteq K(H),$$  then $A$ is  called {\it GCR} or {\it type I}.
For more detailed information on CCR and type I C*-algebras the reader is referred to [\cite{Blackadar}, section IV.1].

In \cite{ContinuitySpectralRadius} it was proved that the class of type I C*-algebras is completely characterized by the behavior of nilpotent elements.
\begin{theorem}
Let $A$ be a C*-algebra. The following are equivalent:

(i) $A$ is type I;

(ii) The spectral radius function $a\mapsto \rho(a)$ is continuous on $A$;

(iii) The closure of nilpotents in $A$ consists of quasinilpotents.

\end{theorem}
\noindent (The implication (i)$\Rightarrow$ (ii) was proved in \cite{ShT}).

Thus in all non-type I algebras there exist non-quasinilpotent limits of nilpotents.
In some non-type I C*-algebras one can even find non-zero normal limits of nilpotents, e.g. in $B(H)$ (\cite{VoiculescuNormal}), in all infinite-dimensional UHF-algebras and in all unital, simple, purely infinite algebras (\cite{Skoufranis}).
However, such normal limits do not exist in arbitrary non-type I C*-algebras, as the following observation from \cite{ContinuitySpectralRadius} shows.

\begin{proposition}\label{NotResiduallyTypeI}
If the closure of nilpotents in a C*-algebra $A$ contains a nonzero normal element, then A is not residually type I.
\end{proposition}

Here a residually type I C*-algebra denotes a C*-algebra that has a separating family of $\ast$-homomorphisms into type I C*-algebras.
The class of residually type I C*-algebras is known to being strictly larger than the class of type I C*-algebras  as, for instance, it contains all residually finite-dimensional (RFD) C*-algebras which include numerous non-type I examples.

\section{Extensions by compact operators}

In this section we consider  extensions of non-zero C*-algebras by compact operators or, in other words,  C*-algebras of operators strictly containing all compact operators.
In the unital case Phillips proved that for such C*-algebras the corresponding central sequence algebras are not abelian (\cite{Phillips}).
Here we obtain a far reaching generalization of this result stating that they are not even residually type I. Moreover, our result covers the non-unital case as well.
Our proof is also more direct in the sense that it doesn't use Connes technique for automorphism groups or the connection between automorphism groups and central sequences in general, which was the main tool in \cite{Phillips}.

The results of this section will allow us to pass from type I to CCR C*-algebras later when we will characterize  subhomogeneity of central sequence algebras.

\medskip

Let $H$ be a separable, infinite-dimensional Hilbert space and denote by $\pi\colon B(H) \to B(H)/K(H)$ the canonical surjection.

\begin{lemma}\label{commutant}
Let $A\subseteq B(H)$ be a separable C*-algebra, then $\pi(A)'$ contains a copy of $B(H)$.
\end{lemma}

\begin{proof}
Let $\iota\colon A\to B(H)$ denote the inclusion map.
Let $\varrho_0\colon A\to A/(A\cap K(H))\to B(H)$ be any non-degenerate representation and set $\varrho = \varrho_0^{\oplus\infty}$, i.e. $\varrho=\varrho_0 \otimes 1\colon A\to B(H\otimes H)$.
By Voiculescu's theorem (\cite{BrownOzawa}, Th. 1.7.3), $\iota$ and $\iota\oplus\varrho$ are unitarily equivalent after modding out the compacts.
In particular, $\pi(A)'$ contains a copy of $\pi((\iota\oplus\varrho)(A)')\supseteq \pi(0\oplus (1\otimes B(H)))\cong B(H)$.
\end{proof}

Next we prove a key step towards Theorem \ref{MainResult} which will  eventually serve as a reduction step from type I C*-algebras to the CCR case.
Our strategy is to obtain approximately central nilpotent elements by lifting suitable nilpotent elements from the Calkin algebra and then absorb them into the C*-algebra of interest using compact cut-downs while approximately preserving commutation relations from the Calkin algebra.
The idea that the relative commutant in the Calkin algebra might be useful in exploring central sequences is somewhat inspired by Theorem 2.7 in \cite{Phillips}.

\begin{theorem}\label{ContainK(H)}
Let $A \subset B(H)$ be a separable C*-algebra such that $A \supsetneq K(H)$. Then $F(A)$ is not residually type I.
\end{theorem}
\begin{proof} At first we observe that the proof of Lemma \ref{commutant} in addition shows that for any $a\in A$ and any $b$ in the copy of $\pi(0\oplus (1\otimes B(H))) $ inside $\pi(A)'$ we have
$\|\pi(a)b\| = \|b\pi(a)\| = \|\pi(a)\|\|b\|.$

Now fix any $a_0\in A\backslash K(H)$ with $\|\pi(a_0)\| =1$.

Since in $B(H)$ there exists a sequence of nilpotents converging to a non-zero normal operator (\cite{VoiculescuNormal}),  by Lemma \ref{commutant} and the observation above we can find nilpotent elements $b_j\in \pi(A)'$, $j\in \mathbb N$,  and a non-zero normal element $b\in \pi(A)'$ such that $b_j\to b$, $\|b_j\| = 1$, and
$\|b_j\pi(a_0)\|=1$ for any $j\in \mathbb N$.
Since the $b_j$'s converge to a normal element, by passing to a subsequence we can further assume that $\|b_j-b_{j-1}\|< 1/2^j$ and $\|[b_j^*, b_j]\|< 1/j$. Let $s_j$ be the order of nilpotence of
$b_j$.

By Theorem \ref{LiftabilityOfNilpotents} we can lift each $b_j$ to a nilpotent contraction $T^{(j)} \in B(H)$ of the same order $s_j$. Then all $T^{(j)}$ obtained this way will commute with $A$ modulo the compacts.
We fix a dense subset $a_1, a_2, \ldots$ in the unit ball of $A$.

Now let $\{k_{\lambda}\}$ be an approximate unit in $K(H)$ which is quasicentral for the C*-algebra $C^*(\it A, T^{(1)}, T^{(2)}, \ldots)$.
We can then choose  $k_{\lambda_n}$, denoted by $k_n$ for short, such that

\begin{align}
\|k_n[T^{(j)}, a_i]- [T^{(j)}, a_i]\| & < 1/n, & i\le n, j\le n, \label{0new} \\
\|[k_n, a_i]\| & < 1/n, & i\le n, \label{0'new} \\
\|[k_n, T^{(j)}]\| & < \frac{2}{s_j(s_j-1)n}, & j\le n, \label{4new} \\
\|(1-k_n)(T^{(j)} - T^{(j-1)})\| & < 1/2^{j}, & 1<j\le n, \label{newnew2}\\
\|[((1-k_n)T^{(j)})^*, (1-k_n)T^{(j)}]\| & < 1/j, & j\le n, \label{newnew3} \\
\|(1-k_n)T^{(j)}a_0\| & \geq 1, & \text{for any}\; j,n. \label{newnew4}
\end{align}

We set $$T^{(j)}_n = (1-k_n)T^{(j)},$$
then by (\ref{0new}), (\ref{0'new}) it follows that for each $n$
\begin{equation}\label{1new}
\|[T^{(j)}_n, a_i]\| = \|[T^{(j)}, a_i] - k_n[T^{(j)}, a_i] + [a_i, k_n]T^{(j)}\| < 2/n, \; i\le n, j\le n.
\end{equation}

\noindent Now let $\{y_{\lambda}\}$ be another approximate unit in $K(H)$ which is quasicentral for the C*-algebra $C^*(A, T^{(1)}, T^{(2)}, \ldots)$.
(We can actually take the same q.a.u. $\{k_{\lambda}\}$, but we choose different notation for convenience.)
Then we can find $y_n$ such that

\begin{equation}\label{2new}
\|[y_n, a_i]\|< 1/n, \; i\le n,
\end{equation}

\begin{equation}\label{3new}
\|[T^{(j)}_n, y_n]\|< \frac{2}{s_j(s_j-1)n}, j\le n,
\end{equation}

\begin{equation}\label{**new}
\|[\left(y_nT_n^{(j)}\right)^*, y_nT_n^{(j)}] \|< \frac{1}{j},  j\le n,
\end{equation}
and, since $y_{\lambda}\to 1$ in the strong operator topology,
\begin{equation}\label{***new}
\|y_nT_n^{(j)}a_0\|\geq 1, j\leq n.
\end{equation}

Let $$x^{(j)}_n = y_nT^{(j)}_n \in K(H)\subset A.$$
Then by (\ref{1new}), (\ref{2new})
$$\|[x^{(j)}_n, a_i]\| = \|y_n[T^{(j)}_n, a_i] + [y_n, a_i]T^{(j)}_n\| < 3/n, \; i,j\le n.$$
Thus $(x^{(j)}_n)_{n\in \mathbb N}$ is a central sequence in $A$ for each $j$.
By (\ref{4new}) and Lemma \ref{trivial} (v) we verify
\begin{equation}\label{5new}
\|\left(T^{(j)}_n\right)^{s_j}\| \le \frac{(s_j-1)s_j}{2} \|[T^{(j)}, 1-k_n]\|< \frac{1}{n}, \; j\le n,
\end{equation}
and by (\ref{5new}),  (\ref{3new}) and Lemma \ref{trivial} (v)
\begin{equation}\label{newnew}
\|\left(x_n^{(j)}\right)^{s_j}\| \le \|y_n^{s_j}(T_n^{(j)})^{s_j}\| + \frac{s_j(s_j-1)}{2}\|[T_n^{(j)}, y_n]\|< \frac{2}{n}, \; j\le n.
\end{equation}

Let $x^{(j)}=(x^{(j)}_n)_{n\in \mathbb N}\in A_\infty$, for each $j\in \mathbb N$. By (\ref{newnew}), each $\pi_{\omega}( x^{(j)})$ is nilpotent.
Using (\ref{newnew2}) we find
$$\|\pi_{\omega} (x^{(j)}) - \pi_{\omega} (x^{(j-1)})\| \le  \frac{1}{2^{j-1}}.$$
which implies that the sequence $\pi_{\omega}(x^{(j)})$ is Cauchy and hence converges to $\pi_{\omega}(x)\in A_{\omega}\bigcap A'$, for some $x= (x_n)$.
By (\ref{**new}),
$$\|[\pi_{\omega}(x^{(j)})^*, \pi_{\omega}( x^{(j)})] \|\le 1/j$$ and hence $\pi_{\omega}( x)$ is normal.
By (\ref{***new}), $\|\pi_{\omega}( x^{(j)}) a_0\|\ge 1$ for every $j\in \mathbb N$ and hence $\|\pi_{\omega}( x) a_0\|\ge 1$ which implies that $\pi_{\omega}( x)\notin Ann(A, A_{\omega})$.
Therefore the images of $ x^{(j)}$ in $F(A)$, $j\in \mathbb N$, form a sequence of nilpotents converging to a non-zero normal element.
By Proposition \ref{NotResiduallyTypeI}, it follows that $F(A)$ is not residually type I.
 \end{proof}

 Recall that a C*-algebra is called {\it residually finite-dimensional} (RFD) if it has a separating family of finite-dimensional representations.

 \begin{corollary} Let $A \subset B(H)$  be a separable C*-algebra such that $A \supsetneq K(H)$.  Then $F(A)$ is not type I and not RFD.
\end{corollary}

\section{Fell's condition of order $n$ and type $I_n$ C*-algebras.}

\subsection{Fell's condition of order $n$}
\hfill\\
\vspace*{-2mm}

\noindent {\bf Definition}(e.g. see \cite{FellProperty}) An irreducible representation $\pi_0$ of a C*-algebra $A$ satisfies {\it Fell's
condition} if there exist $b\in A^+$ and an open neighbourhood $U$ of $\pi_0$ in $\hat A$
 such that $\pi(b)$ is a
rank-one projection whenever $\pi \in U$.

\medskip

\noindent {\bf Definition} (\cite{Blackadar}, \cite{FellProperty}) A C*-algebra $A$ is said to satisfy {\it Fell's condition} (also is called {\it Fell algebra}) if every irreducible
representation of $A$ satisfies Fell's condition.

\begin{remark}\label{RemarkProjection} In the definition of Fell's condition one can replace the requirement that $\pi(b)$ is a
projection by the requirement that $\pi_0(b)\neq 0$. Indeed suppose there exist $a\in A^+$ and an open neighbourhood $U$ of $\pi_0$ in $\hat A$
 such that $\pi_0(a)\neq 0$ and $\pi(a)$ has rank one whenever $\pi \in U$. Passing to a smaller neighborhood, we can assume  that the spectrum of $\pi(a)$ consists of $0$ and a number bigger  than $\|\pi_0(a)\|/2$ whenever $\pi \in U$. Let $f$ be a function which vanishes at $0$ and satisfies $f(t)=1$ for any $t\ge \|\pi_0(a)\|/2$. Let $b= f(a)$. Then $\pi(b)$ is a
rank-one projection whenever $\pi \in U$.
\end{remark}

We now introduce Fell's condition of higher order.

\medskip

\noindent {\bf Definition} An irreducible representation $\pi_0$ of $A$ satisfies {\it Fell's
condition of order $k$} if
there exist $b\in A^+$ and an open neighbourhood $U$ of $\pi_0$ in $\hat A$
 such that $\pi_0(b)\neq 0$ and $\pi(b)$ has rank not larger than $k$ whenever $\pi \in U$.

\medskip

\begin{remark} We were recently informed by Douglas Somerset that Fell's condition of  higher order is closely related with the notion of upper multiplicity introduced by Archbold \cite{Archbold}, namely $\pi$ satisfies Fell's condition of order $k$ if and only if the upper multiplicity of $\pi$ is not larger than $k$ (see  Th.2.5 in \cite{ArchboldSomersetSpielberg}).
\end{remark}

\noindent {\bf Definition} A C*-algebra $A$ is said to satisfy {\it Fell's condition of order $k$} if every irreducible
representation of $A$ satisfies Fell's condition of order $k$.

\begin{example} The C*-algebra
$$A = \{f\in C([0, 1], M_2)\;|\; f(1) \;\text{is diagonal}\}$$
satisfies Fell's condition of order 1 (in other words, it is a Fell algebra).

\medskip

Indeed, $A$ has two-dimensional irreducible representations $ev_t$ given by the evaluation at $t \in [0, 1)$ and two one-dimensional representations $\pi_1, \pi_2$ corresponding to $t=1$.
It is easy to see that all of them satisfy Fell's condition.
\end{example}

\begin{example} The C*-algebra
$$A = \{f\in C([0, 1], M_3)\;|\; f(1) = \left(\begin{array}{ccc} \lambda & & \\ & \lambda &\\ &&\mu\end{array}\right),  \lambda, \mu\in \mathbb C\}$$
satisfies Fell's condition of order 2 but not of order 1.

\medskip

Indeed, $A$ has three-dimensional irreducible representations $ev_t$ given by the evaluation at $t \in [0, 1)$ and two 1-dimensional representations $\pi_1, \pi_2$ corresponding to $t=1$ given by
$$\pi_1(x)=(x(1))_{11} = (x(1))_{22}, \;\; \pi_2(x)=(x(1))_{33}.$$  All the irreducible representations but $\pi_1$  satisfy Fell's condition, but $\pi_1$ does not, because for any $b\in A$ with $\pi_1(b)\neq 0$ we have $rank \; b(1) \ge 2$ and by the lower semicontinuity of the rank $$\liminf_{t\to 1} rank \; ev_t (b) = \liminf_{t\to 1} rank \; b(t) \ge rank \; b(1) \ge  2.$$ Thus $A$ does not satisfy Fell's condition. On the other hand, it is easy to see that it satisfies Fell's condition of order 2.
\end{example}

\begin{example} Let $T\left(M_{2^{\infty}}\right)$ be the telescopic C*-algebra associated with the UHF algebra
$$\mathbb C\subset M_2\subset M_4 \subset \ldots \subset  M_{2^{\infty}},$$ that is
$$T\left(M_{2^{\infty}}\right) = \{f\in C_0((0, \infty], M_{2^{\infty}})\;|\;  t\le i \Rightarrow  f(t)\in M_{2^i}\}.$$
  Let $$A = \{f\in T\left(M_{2^{\infty}}\right)\;|\; f(\infty) \in \mathbb C 1\}.$$ Then $A$ is CCR but does not satisfy Fell's condition of order $k$, for any $k\in \mathbb N$.

  \medskip

   Indeed, to see that $A$ is CCR we notice that any its irreducible representation, except for the one-dimensional representation  corresponding to the point $\infty$, must not vanish on the ideal of all functions vanishing at infinity and hence not vanish at the ideal of all functions vanishing at $[n, \infty]$, for some $n$, and therefore must be the evaluation at some point, hence is finite-dimensional.

  \noindent To see that $A$ does not satisfy Fell's condition of any order, we notice that the irreducible representation $\pi_{\infty}$ corresponding to $\infty$ (that is a representation given by the formula $\pi_{\infty}(x) =(x(\infty))_{11}$) does  not satisfy Fell's condition of any order. Indeed  let $b\in A$  with $\pi_{\infty}(b)\neq 0$. Then $b(\infty)$ is invertible and hence $b(t)$ is invertible (and hence is of full rank) for all sufficiently big $t$.  This implies that in any neighborhood of $\pi_{\infty}$ there are representations which take values at $b$ of arbitrary big rank.

  \end{example}

One can wonder whether in the definition of Fell's condition of higher order one can additionally require $\pi(b)$ to be a projection, analogously to Fell's condition of order 1 (see Remark \ref{RemarkProjection}). It turns out that it is not the case anymore when the order is bigger than 1, as the following example shows.

\begin{example} Let $\overline S$ be the topologists's sine curve, that is the closure of
$$S = \{(x, sin \frac{1}{x})\;|\; x \in (0, 1]\}.$$
Let $$A = \{f\in C(\overline S, M_2) \;|\; \exists \; g \in C_0[-1, 1) \; \text{s. t.} \; f((0, t)) = \left(\begin{array}{cc} g(t) &\\& g(0)\end{array}\right), \forall t\in [-1, 1]\}.$$
Then for any $f\in A$, $f((0, 0)) \in \mathbb C 1_2.$ Let $\pi = (ev_{(0,0)})_{11} \in \hat A$. We notice that for any $(x_n, y_n)\in S$ such that $(x_n, y_n) \to (0, x)$, and any $f\in A$ such that $f((x_n, y_n)) =0$ we have $$0=f((0, x)) = \left(\begin{array}{cc} g(x) &\\& g(0)\end{array}\right),$$ which implies that $f((0, 0)) = \left(\begin{array}{cc} g(0) &\\& g(0)\end{array}\right) =0.$ In other words,
\begin{equation}\label{ExampleObs} ev_{(x_n, y_n)} \to \pi \end{equation} in $Prim A$, which coincides with $\hat A$ since $A$ is subhomogeneous.

Now, since $A$ is $2$-subhomogeneous, it satisfies Fell's condition of order 2. However we will show that there is no $h\in A^+$ which locally is a non-zero projection (of rank nor larger than 2) around $\pi$. Assume, for the sake of contradiction, that there is  $h\in A^+$ and a neighborhood $U$ of $\pi$ such that $\rho(h)$ is a non-zero projection, for any $\rho\in U$. Let $\rho_n = ev_{(\frac{1}{\pi/2 + 2\pi n}, 1)}$ and $\rho'_n = ev_{(\frac{1}{2\pi n}, 0)}$.  By (\ref{ExampleObs}) and by the assumption, there is $N$ such that  $\rho_n(h)$ and $\rho'_n(h)$ are non-zero projections, for any $n>N$. Since $$\rho_n(h) = h((\frac{1}{\pi/2 + 2\pi n}, 1)) \to h((0, 1)) = \left(\begin{array}{cc} \ast &\\& 0\end{array}\right),$$ it follows that $$rank \; h((\frac{1}{\pi/2 + 2\pi n}, 1)) = 1$$ for sufficiently large $n>N$. Since $$\rho'_n(h) = h((\frac{1}{2\pi n}, 0)) \to h((0, 0))\in \mathbb C 1_2,$$ it follows that $$rank \; h((\frac{1}{2\pi n}, 1)) = 2$$ for sufficiently large $n>N$, which is impossible, because for any $n>N$, the continuous path of the projections $h((t, sin \frac{1}{t}))$, $t\in [\frac{1}{\pi/2 + 2\pi n}, \frac{1}{2\pi n}]$, connects $h((\frac{1}{\pi/2+2\pi n}, 0))$ with $h((\frac{1}{2\pi n}, 1))$. Contradiction.

\end{example}

\subsection{Type $I_n$ algebras}

Let $A$ be a C*-algebra. Recall that an element $x\in A$ is {\it abelian}
if the hereditary C*-subalgebra $\overline{x^*Ax}$ is commutative (\cite{Blackadar}).
Equivalently (\cite{Blackadar}, Prop. IV.1.1.7), $x$ is abelian if and only if  $rank \; \pi(a) \le 1$ for each irreducible representation $\pi$ of $A$.
$A$ is called {\it  type $I_0$} if it is generated by its  abelian elements (\cite{Blackadar}).

Now we introduce C*-algebras of higher type.

\medskip

\noindent {\bf Definition} An element $x\in A$ has {\it global rank not larger than $n$} if $rank \; \pi(a) \le n$ for each irreducible representation $\pi$ of $A$.

 \medskip

\noindent {\bf Definition} A C*-algebra $A$ is {\it type $I_n$} if it is generated  by elements of global rank not larger than $n+1$.

\medskip

\noindent (In fact "generated" can be replaced by "spanned" as the rank of product is always not larger than the minimum of the ranks of multipliers).

Clearly any type $I_n$ algebra is CCR.


\subsection{Fell's condition of order $n$ is equivalent to being type $I_{n-1}$}
The proof below goes along the lines of the proof of the fact [\cite{FellProperty}, Th. 3.3]
that Fell's property is equivalent to being type $I_0$.

\begin{lemma}\label{CharacterizationOfFell} Let $A$ be  a C*-algebra and $\pi \in \hat A$. The following are equivalent:

(i) $\pi$ satisfies Fell's condition of order $n$;

(ii) there exists $a\in A^{+}$ of global rank not larger than $n$ such that $\pi(a)\neq 0$.
\end{lemma}
\begin{proof} (i) $\Rightarrow$ (ii): Since $\pi$ satisfies Fell's condition of order  $n$, there is $b\in A$ and a neighborhood $U$ of $\pi$ such that $\pi(b)\neq 0$ and  $ rank \;\rho(b) \le n$, for any $\rho\in U$. Let
  $$I = \bigcap_{\sigma\notin U} \ker \sigma.$$ Since $\pi\in U$, $\pi$ does not vanish on $I$. Then for any approximate unit $\{i_{\lambda}\}$ in $I$, $\pi(i_{\lambda})$ converges to the identity in the strong operator topology. Hence we can choose $\lambda$ such that $\pi(i_{\lambda}bi_{\lambda}) \neq 0.$ Let $a= i_{\lambda}bi_{\lambda}$. Then $rank \; \sigma(a) = 0$, for any $\sigma \notin U$, and $rank \; \rho(a) \le rank \; \rho(b) \le n,$ for any $\rho\in U$. Hence $a$ has global rank not larger than $n$ and $\pi(a)\neq 0.$

  (ii) $\Rightarrow$ (i): By taking $b = a$ and $U = \hat A$, we see that (i) holds.
  \end{proof}

\begin{proposition}\label{Fell}  A C*-algebra $A$ satisfies Fell's condition of order $n$ if and only if  $A$ is type $I_{n-1}$.

\end{proposition}
\begin{proof}
{\bf "Only if"}: Let $J\subseteq A$ be the $C^*$-algebra generated by all elements of global rank not larger than $n$. By Lemma \ref{CharacterizationOfFell}, no irreducible representation of $A$ vanishes on $J$. Since  $J$ is in fact an ideal, it follows that $J=A$.

  {\bf "If"}: if  $A$ is type $I_{n-1}$, then for any $\pi\in \hat A$  there is $a\in A^{+}$ of global rank not larger than $n$ such that $\pi(a) \neq 0$. By Lemma \ref{CharacterizationOfFell} we are done.
  \end{proof}

\section{$A$ is CCR and $F(A)$ is $k$-subhomogeneous  $\Rightarrow$ $A$ satisfies Fell's condition of order
 $k$.}

\begin{theorem}\label{SubhomogeneousImpliesFell}
Let $A$ be a  separable CCR C*-algebra. If $F(A)$ is $k$-subhomogeneous, then $A$ satisfies Fell's condition of order
 $k$.
\end{theorem}

Below  we use notation $\{e_{ij}\}$ for matrix units in any matrix algebra.
\begin{proof}
Assume $\pi \in\hat{A}$ does not satisfy Fell's condition of order $k$ and let $h\in A$ be a strictly positive element with $\|\pi(h)\|=1$.
By Theorem \ref{Fact}, it will be sufficient to construct  a nilpotent element of order $k+1$ but not of order $k$ in $F(A)$.
Given a finite subset $\mathcal{G}\subset A$ and $\epsilon>0$, it suffices to find $x\in A_1$ with $\|x^{k+1}\|\le \epsilon$,  $\|x^k h\| \geq 1-\epsilon$ and $\|[x,\mathcal{G}]\|\leq\epsilon$.

Without loss of generality we may assume that $\mathcal{G}\subset A_m:=\overline{(h-1/m)_+A(h-1/m)_+}$ for some $m>2/\epsilon$ and that $(h-1/m)_+\in \mathcal{G}$.
Since $A$ is assumed to be CCR, $\pi(h)$ is compact and hence $\pi((h-1/m)_+)$ has finite rank.
We can therefore regard the restriction $\pi|_{A_m}$ as a finite-dimensional representation $A_m \to M_d$ for some $d$.
We may clearly further assume that $\pi|_{A_m}((h-1/m)_+)$ is diagonal with
\[
(\pi|_{A_m}((h-1/m)_+))_{11} = \|\pi|_{A_m}((h-1/m))_+\| = \|\pi((h-1/m)_+)\|=1-1/m.
\]

The reduction to the finite-dimensional case allows us to employ the non-commutative Urysohn lemma ([\cite{ELP}, Th.2.3.3]) as a technical tool to parameterize a neighbourhood of $\pi$.
Concretely, we get a commutative diagram with exact rows
\[
\xymatrix{
0 \ar[r] & J \ar[r] & A_m \ar[r]^{\pi_{|A_m}} & M_d \ar[r] & 0 \\
0 \ar[r] & S(\mathbb{C},M_d) \ar[u]^\alpha \ar[r] & T(\mathbb{C},M_d) \ar[r]^{ev_2} \ar[u]^{\overline{\alpha}}& M_d \ar[r] \ar@{=}[u] & 0
}
\]
where $$T(\mathbb{C},M_d)=\{f\in C_0((0,2],M_d)\colon f(t)\in\mathbb{C}\cdot 1_d \;\text{for all}\; t\leq 1\}$$  and
$$S(\mathbb{C},M_d) = \{f\in C_0((0,2),M_d)\colon f(t)\in\mathbb{C}\cdot 1_d \;\text{for all}\; t\leq 1\},$$ with a proper *-homomorphism $\alpha$ (meaning that the hereditary C*-subalgebra $her_J(image(\alpha))$ of $J$ generated by the image of $\alpha$ equals all of $J$).

Denote by $f_1$ the scalar function on $[0,2]$ which vanishes on $[0,1]$, satisfies $f_1(2)=1$ and is linear in between.
We can thus decompose each element $g\in \mathcal G$ as
$$g=g_0+\overline{\alpha}(g_1)$$
with $g_0\in J$ by defining $$g_1=f_1\otimes \pi|_{A_m}(g)\in T(\mathbb{C},M_d).$$
Let $$B_n = \{f\in S(\mathbb{C},M_d)\;|\; supp \; f \subseteq [0,2-1/n]\}.$$ Then $$S(\mathbb{C},M_d) = \overline{\bigcup_n B_n}$$ and
$J=her\left(\alpha(S(\mathbb{C},M_d)))\right)=\overline{\bigcup_n her(\alpha(B_n)}).$ Therefore, since $G$ is finite,  there is $n\in \mathbb N$ such that for each $g\in \mathcal{G}$ there is  $g'_0\in her(\alpha(B_n))$  such that
\begin{equation}\label{SubhomogeneousImpliesFellNewEq1}\|g_0-g_0'\|\le \epsilon.\end{equation} We can also assume that \begin{equation}\label{SubhomogeneousImpliesFellNewEq2} \frac{1}{n}\le \epsilon.\end{equation}
In particular, using the scalar function $f_n$ which vanishes on $[0,2-1/n]$, satisfies $f_n(2)=1$ and is linear in between, for each $g\in \mathcal{G}$ we conclude that $g_0'$ and $\overline{\alpha}(f_n\otimes 1_d)$ are orthogonal to each other.

Let $y:=\overline{\alpha}(f_n\otimes e_{11})\in A_m \subset A$.
Since $\pi(y) = e_{11}\neq 0$ and by assumption $\pi$ does not satisfy Fell's condition of order $k$,
there is an irreducible representation $\rho \in \hat A$ such that
\[
rank(\rho(y))\geq k+1.\]
We consider the restriction $\rho|_{A_m}$, which can again be regarded as a finite-dimensional, irreducible representation $A_m \to M_{d'}$ for suitable $d'$.
Hence the composition $\rho|_{A_m}\circ \overline{\alpha}$ decomposes as
\[
\rho|_{A_m}\circ \overline{\alpha} \sim \oplus_{j\in J} ev_{t_j}
\]
with finite index set $J\subset [0,2]$.
Since the support of $f_n\otimes e_{11}$ is contained in $[2-1/n,2]$, the rank condition above implies that  $J$ contains at least $k+1$ many (not necessarily different) $t_j$'s with $t_j\in(2-1/n,2]$.
After unitary conjugation we can therefore write
\[
\rho|_{A_m}\circ \overline{\alpha} = (\bigoplus_{j=1}^{k+1}ev_{t_j})\oplus\tilde\varrho
\]
with $t_1,...,t_{k+1}$ in $(2-1/n,2]$. This implies that for each $g=g_0+\overline{\alpha}(g_1)\in \mathcal{G}$ we find
\begin{equation}\label{SubhomogeneousImpliesFellEq1}
\rho|_{A_m}(\overline{\alpha}(g_1)) = (\oplus_{j=1}^{k+1} g_1(t_j)) \oplus \tilde\varrho(g_1) = \oplus_{j=1}^{k+1}((t_j-1)\pi|_{A_m}(g)) \oplus \tilde\varrho(g_1).
\end{equation}
Moreover, as $\rho|_{A_m}(\overline{\alpha}(f_n\otimes 1_d)) = (\oplus_{j=1}^{k+1} f_n(t_j)1_d) \oplus \tilde\varrho(f_n)$ with each $f_n(t_j)>0$, orthogonality of $g_0'$ and $\overline{\alpha}(f_n)$ implies that in the $2\times 2$ block decomposition of $\rho|_{A_m}(g_0')$  the only possibly non-zero block is the $(2,2)$-block. Combining this with (\ref{SubhomogeneousImpliesFellEq1}), (\ref{SubhomogeneousImpliesFellNewEq1}) and (\ref{SubhomogeneousImpliesFellNewEq2}) we find   each $\rho|_{A_m}(g)$
to be of the form
 \[
\rho|_{A_m}(g) = \left((\pi|_{A_m}(g))^{\oplus k+1} \oplus r_{g}\right) + S_{g}
\]
for some $r_{g}\in M_{d'-(k+1)d}$ and $S_{g}\in M_{d'}$ such that $\|S_{g}\|\le 2 \epsilon$.

Consequently, the element $z= \sum_{j=1}^k 1_d\otimes e_{j(j+1)}\in M_{d'}$ is a nilpotent  contraction of order $k+1$ and commutes with $\rho(\mathcal{G})$ up to $2\epsilon$.
By \cite{LoringShulmanNilpVar} we can
lift $z$ to a nilpotent contraction $x$ of order $k+1$ in $A_m$   that commutes with $\mathcal{G}$ up to $2\epsilon$.
Since $(h-1/m)_+$ is also an element of $\mathcal{G}$, we further verify
\begin{align*}
\|x^kh\| & \geq \|x^k(h-1/m)_+\| - 1/m
 \geq \|\rho|_{A_m}(x^k(h-1/m)_+)\| - 1/m \\
& \ge \|(1_d\otimes e_{1,k+1}) ((\pi|_{A_m}((h-1/m)_+))^{\oplus k+1} )\| - \epsilon -1/m \\
& = \|\pi|_{A_m}((h-1/m)_+)\| -\epsilon - 1/m \\
& = (1-1/m) - \epsilon - 1/m  > 1- 2\epsilon.
\end{align*}
\end{proof}

\section{$A$ is type $I_n$ $\Rightarrow$ $F(A)$ is $(n+1)$-subhomogeneous}

The first lemma is trivial.
\begin{lemma}\label{trivial} Let $x, a$ be contractions in a C*-algebra and $\delta \ge 0$. Then for any  $k\in \mathbb N$ the following holds:

(i) $\|[x, a]\|\le \delta$ ${\Rightarrow}$ $\|[x, a^k]\| \le k\delta$;

(ii) $\|[x, a]\|\le \delta$  ${\Rightarrow}$ $(axa)^k = A_k + ax^ka^{2k-1},$ where $\|A_k\| \le k(k-1)\delta; $

(iii) $\|[x, a]\|\le \delta$ and $x^k=0$ ${\Rightarrow}$ $\|(axa)^k\| \le k(k-1)\delta;$

(iv) $\|x-a\|\le \delta$ $\Rightarrow$ $\|x^k - a^k\| \le k\delta$;

(v) $\|[x, a]\|\le \delta$  ${\Rightarrow}$ $\|(xa)^k\| \le \|x^ka^k\| + \frac{(k-1)k}{2}\delta$.
\end{lemma}
\begin{proof} (i) By induction: suppose for $k\le l-1$ it is proved. For $k=l$ we have
$[x, a^l] = [x, a^{l-1}]a + a^{l-1}[x, a]$ and therefore $\|[x, a^l]\| \le (l-1)\delta + \delta = l\delta$.

(ii) By induction: For $k=1$ it is clear. Suppose it is proved for $k\le N$. We have
\begin{multline*} (axa)^{N+1} = (axa)^Naxa = (A_N+ ax^Na^{2N-1})axa \\ = A_Naxa + ax^N[a^{2N},x]a + ax^{N+1}a^{2N+1} = A_{N+1} + ax^{N+1}a^{2N+1},
\end{multline*}
where by (i)  $$\|A_{N+1}\|\le N(N-1)\delta + 2N\delta = (N+1)N\delta.$$

(iii) follows from (ii).

(iv) By induction: suppose it is proved for $k\le N$. We have $$\|x^{N+1} - a^{N+1}\|\le \|(x-a)x^N\| + \|a(x^N-a^N)\|\le \delta + N\delta = (N+1)\delta.$$

(v) We will show by induction that $(xa)^k =x^ka^k  + E_k$, where $\|E_k\| \le  \frac{(k-1)k}{2}\delta$. For $k=1$ it holds, and we assume it holds for $k\le N$. Then for $k=N+1$ we have
$$(xa)^{N+1} = xax^Na^N + xaE_N =  x^{N+1}a^{N+1} + x[a, x^N]a^N + xaE_N.$$ By this and (i) we obtain
\begin{multline*}\|(xa)^{N+1}\|\le \| x^{N+1}a^{N+1}\| + \|[a, x^N]\| + \|E_N\| \le \| x^{N+1}a^{N+1}\| + (N+ \frac{(N-1)N}{2})\delta \\ \le \| x^{N+1}a^{N+1}\| +  \frac{(N+1)N}{2}\delta.\end{multline*}

\end{proof}

\begin{lemma}\label{Main01} For any $\epsilon > 0$ there is a $\delta >0$ such that whenever $e\in (B(H))_{+, 1} $ is of rank not larger than $N$ and $x\in (B(H))_1$ with $x^{N+1}=0$, then $$\|[x, e]\|\le \delta \;\;\; \Rightarrow \;\;\; \|ex^N\|\le \epsilon \;\text{and}\; \|x^Ne\|\le \epsilon.$$
\end{lemma}
\begin{proof} By stability of nilpotent contractions (Corollary \ref{StabilityOfNilpotents}) there exists $\delta_0$ such that whenever $z$ is an element of any $C^*$-algebra with $\|z^{N+1}\|\leq\delta_0$ and $\|z\|\le 1$, there is an element $y$ in the same C*-algebra with $y^{N+1}=0, \|y\|\le 1, \|y-z\|\le \epsilon/2N.$
Let $$\delta_1 = \min\{\frac{\delta_0}{N(N+1)}, \frac{\epsilon}{2N(3N-2)}\}$$ and let
$$\delta = \left(\frac{4\delta_1}{5}\right)^{2N}.$$
By \cite{PedersenInequality}, the assumption $\|[x, e]\|\le \delta $ implies that $\|[x, e^{\frac{1}{2N}}]\|\le \delta_1.$
By the statement (iii) of Lemma \ref{trivial}
$$\|\left(e^{\frac{1}{2N}}xe^{\frac{1}{2N}}\right)^{N+1}\| \le N(N+1)\delta_1\le \delta_0.$$
Therefore there is $y\in e^{\frac{1}{2N}}B(H)e^{\frac{1}{2N}}$ such that
\begin{equation}\label{MainLemma1}y^{N+1}=0, \|y\|\le 1, \|y-e^{\frac{1}{2N}}xe^{\frac{1}{2N}}\|\le \epsilon/2N.\end{equation}
Since $y\in e^{\frac{1}{2N}}B(H)e^{\frac{1}{2N}}$, it is a nilpotent on a Hilbert space of dimension  $ rank  \; e \le N$ and we conclude that
\begin{equation}\label{MainLemma2}y^N=0.\end{equation}
By (\ref{MainLemma1}), (\ref{MainLemma2}) and the statement (iv) of Lemma \ref{trivial}
\begin{equation}\label{MainLemma3} \|\left(e^{\frac{1}{2N}}xe^{\frac{1}{2N}}\right)^N\| = \|\left(e^{\frac{1}{2N}}xe^{\frac{1}{2N}}\right)^N - y^N\| \le N\epsilon/2N = \epsilon/2.\end{equation}
By the statement (ii) of Lemma \ref{trivial}
\begin{equation}\label{MainLemma4} \left(e^{\frac{1}{2N}}xe^{\frac{1}{2N}}\right)^N = A + e^{\frac{1}{2N}}x^Ne^{\frac{2N-1}{2N}},\end{equation}
where $\|A\| \le N(N-1)\delta_1.$
By (\ref{MainLemma3}) and (\ref{MainLemma4})
\begin{equation}\label{MainLemma6} \|e^{\frac{1}{2N}}x^Ne^{\frac{2N-1}{2N}}\|\le \|A + e^{\frac{1}{2N}}x^Ne^{\frac{2N-1}{2N}}\| + \|A\| \le \epsilon/2  + N(N-1)\delta_1.\end{equation} Now using (\ref{MainLemma6}) and the statement (i) of Lemma \ref{trivial} we obtain
$$\|ex^N\| = \|e^{\frac{1}{2N}}x^Ne^{\frac{2N-1}{2N}} - e^{\frac{1}{2N}}[x^N, e^{\frac{2N-1}{2N}}]\| \le \epsilon/2 + N(N-1)\delta_1 + N(2N-1)\delta_1 \le \epsilon.$$
Similarly $\|x^Ne\|\le \epsilon.$

\end{proof}

\begin{theorem}\label{Main1} If $A$ is type $I_n$, then $F(A)$ is (n+1)-subhomogeneous.
\end{theorem}
\begin{proof} Let $x\in F(A)$ be a contraction such that $x^{n+2}=0$. By liftability of nilpotent contractions  (Theorem \ref{LiftabilityOfNilpotents}), we can represent $x$ by a sequence $(x_1, x_2, \ldots)$ with $x_k^{n+2}=0$, $\|x_k\|\le 1$, $k\in \mathbb N$.

Fix a contraction $a\in A$ of global rank not larger than $n+1$ and let $\epsilon >0$. Let $\delta$ be as in Lemma \ref{Main01}. As
$$\lim_{\omega} [x_k, (aa^*)^{\frac{1}{2}}] = \lim_{\omega} [x_k, (a^*a)^{\frac{1}{2}}] = 0, $$ there is $E\in \omega$ such that for any $k\in E$ we have $$\|[x_k, (aa^*)^{\frac{1}{2}}]\|\le \delta \; \text{and} \; \|[x_k, (a^*a)^{\frac{1}{2}}]\|\le \delta.$$ Since $(aa^*)^{1/2}$ and
 $(a^*a)^{1/2}$ are positive contractions and have global rank not larger than $n+1$, by Lemma \ref{Main01} for any irreducible representation $\rho$ of $A$ and any $k\in E$ we have
$$\|\rho(x_k^{n+1})\rho((aa^*)^{1/2})\|\le \epsilon \; \text{and}\;  \|\rho((a^*a)^{1/2})\rho(x_k^{n+1})\|\le \epsilon.$$ Let $\rho(a) = \rho((aa^*)^{1/2})V_{\rho} = W_{\rho}\rho((a^*a)^{1/2})$, where $V_{\rho}, W_{\rho}\in B(H_{\rho})$, be the polar decomposition of $\rho(a)$. Then
\begin{multline*}\|ax_k^{n+1}\| = \sup_{\rho} \|\rho(a)\rho(x_k^{n+1})\| = \sup_{\rho}   \|W_{\rho}\rho((a^*a)^{1/2})\rho(x_k^{n+1})\| \\ \le \|\rho((a^*a)^{1/2})\rho(x_k^{n+1})\|\le \epsilon,\end{multline*} for any $k\in E$.  Similarly $\|x_k^{n+1}a\| \le \epsilon,$ for any $k\in E$. It follows that the set $\{k\in \mathbb N\;|\; \|ax_k^{n+1}\|\le \epsilon \;\text{and}\; \|x_k^{n+1}a\|\le \epsilon \}$ contains $E$ and therefore belongs to $\omega$.
Thus $\lim_{\omega} \|ax_k^{n+1}\| = \lim_{\omega} \|x_k^{n+1}a\|= 0.$ Since this holds for any  contraction $a\in A$ of global rank not larger than $n+1$ and these span $A$ by assumption, we conclude that the image of the sequence $(x_1^{n+1}, x_2^{n+1}, \ldots)$ in $A^{\omega} \bigcap A'$ belongs to $Ann(A, A^{\omega})$. Hence $x^{n+1}=0.$
By Theorem \ref{Fact}, $F(A)$ is $(n+1)$-subhomogeneous.
\end{proof}

  \section{When central sequence algebras are subhomogeneous}

  \begin{theorem}\label{MainResult} Let $A$ be a separable C*-algebra and $k\in \mathbb N$. Then $F(A)$ is $k$-subhomogeneous if and only if $A$ satisfies Fell's condition of order $k$.
  \end{theorem}
  \begin{proof}
{\bf "Only if"}: In [\cite{AndoKirchberg}, Th.1.1] Ando and Kirchberg proved that if $A$ is a non-type I C*-algebra, then $F(A)$ is not abelian. We note that their proof in fact shows even  that  $F(A)$ is not type I, because they proved that it contains a type III factor as a subquotient [\cite{AndoKirchberg}, p.7].

Hence the assumption that $F(A)$ is $k$-subhomogeneous implies that $A$ is type I. If $A$ was not CCR, then it would quotient onto a C*-algebra $B\subseteq B(H)$ such that $B\supset K(H)$. By Theorem \ref{ContainK(H)}, $F(B)$ is not residually type I and hence not $k$-subhomogeneous. By [\cite{Kirchberg}, Rem. 1.15 (3)],  $F(A)$ would not be $k$-subhomogeneous either. Thus $A$ must be CCR.
  The statement  follows now from  Theorem \ref{SubhomogeneousImpliesFell}.

  {\bf "If"}: By Theorem \ref{Main1}   and the "only if" part of Proposition \ref{Fell}.
  \end{proof}

  Now we can answer a question raised by J. Phillips (\cite{Phillips}) and by Ando and Kirchberg (\cite{AndoKirchberg}).

   \begin{corollary}\label{hypercentral} Let $A$ be a  separable C*-algebra. The following are equivalent:

 (1) $F(A)$ is abelian;

 (2) All central sequences in $A$ are hypercentral;

 (3) $A$ satisfies Fell's condition.
 \end{corollary}
 \begin{proof} Follows from Theorem \ref{MainResult}, Proposition \ref{epsilon-delta} and the paragraph preceding Proposition \ref{epsilon-delta}.
 \end{proof}

  \medskip

  Recall that an automorphism $\alpha$ of a unital C*-algebra $A$ is {\it inner} if there is a unitary $u\in A$ such that
  $\alpha(a) = u^*xu$, for any $a\in A$. An automorphism is {\it approximately inner} if it is a pointwise limit of inner automorphisms.
  The set of all approximately inner automorphisms of $A$ will be denoted by $\overline{Aut \;A}$.

  An automorphism $\alpha$ is {\it centrally trivial} if $\|\alpha(x_n) - x_n\| \to 0,$ for every central sequence $\{x_n\}$ of $A$ (\cite{Phillips}).
  The set of all centrally trivial automorphisms of $A$ is denoted by $Ct \;A$.

  \begin{corollary} Let $A$ be a unital separable C*-algebra that satisfies Fell's condition. Then any approximately inner automorphism of $A$ is centrally trivial.
  \end{corollary}
  \begin{proof} By [\cite{Phillips}, Lemma 3.5] if all central sequences in $A$ are hypercentral, then $\overline{Inn \;A}\subseteq Ct \;A$. The statement follows  now from Corollary \ref{hypercentral}.
  \end{proof}

  \begin{corollary} If a C*-algebra $A$ satisfies Fell's condition but does not have continuous trace, then $A$ has an outer derivation.
  \end{corollary}
  \begin{proof} Follows from [\cite{Lupini}, Prop.5.5] and Corollary \ref{hypercentral}.
  \end{proof}

\end{document}